\documentclass[10pt]{amsart}

\usepackage{amsmath,xspace,amssymb,mathrsfs}
\usepackage{color}

\input xy
\xyoption{all}
\xyoption{2cell}
\UseAllTwocells
\CompileMatrices

\newcommand{\Spec}{\operatorname{Spec}}
\renewcommand{\phi}{\varphi}

\newcommand{\Ker}{\operatorname{Ker}}

\newcommand{\Max}{\operatorname{Max}}

\newcommand{\dg}{\operatorname{deg}}

\newtheorem{proposition}{Proposition}[section]
\newtheorem{lemma}[proposition]{Lemma}

\newtheorem{corollary}[proposition]{Corollary}
\newtheorem{theorem}[proposition]{Theorem}

\theoremstyle{definition}

\newtheorem{example}[proposition]{Example}

\newtheorem{remark}[proposition]{Remark}

\usepackage{etoolbox}
\makeatletter
\patchcmd{\@settitle}{\uppercasenonmath\@title}{}{}{}
\patchcmd{\@setauthors}{\MakeUppercase}{}{}{}
\makeatother

\begin{document}

\title[Grothendieck group applications]{Grading of homogeneous localization by the Grothendieck group}

\author[A. Tarizadeh]{Abolfazl Tarizadeh}
\address{Department of Mathematics, Faculty of Basic Sciences, University of Maragheh, Maragheh, East Azerbaijan Province, Iran.}
\email{ebulfez1978@gmail.com}

\date{}

\begin{abstract} The main result of this article is a fantastic generalization of a classical result in graded ring theory. In fact, our result states that if $S$ is a multiplicative set of homogeneous elements of an $M$-graded commutative ring $R=\bigoplus\limits_{m\in M}R_{m}$ with $M$ a commutative monoid, then the localization ring $S^{-1}R=\bigoplus\limits_{x\in G}(S^{-1}R)_{x}$ is a $G$-graded ring where $G$ is the Grothendieck group of $M$ and each homogeneous component $(S^{-1}R)_{x}$ is the set of all fractions $f\in S^{-1}R$ such that $f=0$ or it is of the form $f=r/s$ where $r$ is a homogeneous element of $R$ and $x=[\dg(r),\dg(s)]$. \\ 
The above result is a natural generalization of the related result in the $\mathbb{N}$-graded case to the more general case of $M$-graded rings. However, the proof of this theorem is more subtle and not the same as the proof of the $\mathbb{N}$-graded case. \\ 
As an application, we effectively bridge the gap between monoid algebra and group algebra using the tool of localization. In fact, we show that there is a canonical isomorphism of ($G$-graded) rings $S^{-1}(R[M])\simeq R[G]$ where $S=\{\epsilon_{m}: m\in M\}$ is a multiplicative set of the monoid-ring $R[M]$ with $\epsilon_{m}=(\delta_{m,n})_{n\in M}$ is a (positional) homogeneous element of $R[M]$ and $\delta_{m,n}$ is the Kronecker delta. \\
In particular, the ring of Laurent polynomials $R[x_{1}^{\pm1},\ldots,x_{n}^{\pm1}]$ is canonically isomorphic to the group-ring $R[\mathbb{Z}^{n}]$ which is just a tiny slice of a much more universal algebraic behavior.  
\end{abstract}

\maketitle

\section{Introduction and Preliminaries}

The Grothendieck group of a commutative monoid is a fundamental construction of mathematics. In fact, several of the basic structures of mathematics derive from Grothendieck groups. For instance, the additive group of integers is indeed the Grothendieck group of the additive monoid of natural numbers. The role of Grothendieck group in some other advanced structures of mathematics is very influential. In this article, we show that the Grothendieck group also plays an important role in graded ring theory. 

It is well known that for any multiplicative set $S$ of homogeneous elements of an $\mathbb{N}$-graded commutative ring $R=\bigoplus\limits_{n\geqslant0}R_{n}$, then the localization ring  $S^{-1}R=\bigoplus\limits_{n\in \mathbb{Z}}(S^{-1}R)_{n}$ is a $\mathbb{Z}$-graded ring where each homogeneous component $(S^{-1}R)_{n}$ is an additive subgroup of $S^{-1}R$ consisting of all fractions $f\in S^{-1}R$ such that $f=0$ or it is of the form $f=r/s$ where $r$ is a homogeneous element of $R$ and $\dg(r)-\dg(s)=n$.

This result is widely used in commutative algebra and algebraic geometry, especially for the Proj construction. 

In this article, we generalize this result to the more general setting of $M$-graded commutative rings where $M$ is a commutative monoid. One of the main motivations for this generalization stems from the fact that for any commutative ring $R$ the monoid ring $R[M]$ is an $M$-graded commutative ring.
 
In fact, our generalization reads as follows. If  $S$ is a multiplicative set of homogeneous elements of an $M$-graded ring $R=\bigoplus\limits_{m\in M}R_{m}$, then the localization ring $S^{-1}R=\bigoplus\limits_{x\in G}(S^{-1}R)_{x}$ is a $G$-graded ring where $G$ is the Grothendieck group of the commutative monoid $M$, and each homogeneous component $(S^{-1}R)_{x}$ is an additive subgroup of $S^{-1}R$ consisting of all fractions $f\in S^{-1}R$ such that $f=0$ or it is of the form $f=r/s$ where $r$ is a homogeneous element of $R$ and $x=[\dg(r),\dg(s)]$.  

As mentioned earlier, the above result is a natural generalization of the related result in the $\mathbb{N}$-graded case to the more general case of $M$-graded rings. However, the proof of this theorem is slightly technical and is not the same as the proof of the $\mathbb{N}$-graded case.  
In fact, this result is proven in two main steps. In the first step (Lemma \ref{Lemma 1-bir}), we prove the assertion for cancellative monoids. In the second step (Theorem \ref{Theorem 20 altin}), by changing the grading we reduce the problem to the cancellative case and then using the first step, the proof is completed. 

Next, we apply this result for some certain localizations of monoid rings. For instance, we show that the localization of the monoid-ring $R[M]$ with respect to the multiplicative set $\{\epsilon_{m}: m\in M\}$ is canonically isomorphic to the group-ring $R[G]$ where $G$ is the Grothendieck group of $M$. In particular, the following well known result is recovered which states that the ring of Laurent polynomials $R[x,x^{-1}]$ is canonically isomorphic to the group-ring $R[\mathbb{Z}]$.

In Theorem \ref{Theorem III} we prove that a cancellative monoid is a totally ordered monoid if and only if its Grothendieck group is a torsion-free group. We learned that a classical version of this theorem can be found in Northcott \cite[\S 2.12, Theorem 22]{Northcott}. However, we obtained this result independently (especially, our proof is completely different and much shorter than Northcott's approach). Then as an application, Levi’s famous theorem (which states that every torsion-free Abelian group is a totally ordered group) is easily deduced. 

Finally, we show that the group of units of any  localization ring $S^{-1}R$ is canonically isomorphic to the Grothendieck group of the saturation of $S$. In particular, if $\mathfrak{p}$ is a prime ideal of a ring $R$ then the Grothendieck group of the multiplicative monoid $R\setminus\mathfrak{p}$ is canonically isomorphic to the group of units of the local ring $R_{\mathfrak{p}}$. As another application, the Grothendieck group of the set of non-zero-divisors of $R$ is canonically isomorphic to the group of units of the total ring of fractions of $R$. 

In this article, all monoids and rings are assumed
to be commutative. Recall that the Grothendieck group of a commutative monoid $M$ is constructed in the following way. We first define a relation over the set $M\times M$ as $(a,b)\sim(c,d)$ if there exists some $m\in M$ such that $(a+d)+m=(b+c)+m$. It can be seen that this is an equivalence relation. Here we denote the equivalence class containing of an ordered pair $(a,b)\in M\times M$ simply by $[a,b]$, and we denote by $G=\{[a,b]: a,b\in M\}$  the set of all equivalence classes obtained by this relation. Then the set $G$ by the operation $[a,b]+[c,d]=[a+c,b+d]$ is an Abelian group. Indeed, $[0,0]$ is the identity element of $G$ where $0$ is the identity of $M$, and for each $[a,b]\in G$ its inverse is $[b,a]$. The group $G$ is called the \emph{Grothendieck group of $M$}. Note that in the above construction, the commutativity of the monoid $M$ plays a vital role. 

The canonical map $f:M\rightarrow G$ given by $m\mapsto[m,0]$ is a morphism of monoids and the pair $(G,f)$ satisfies in the following universal property: for any such pair $(H,g)$, i.e., $H$ is an abelian group and $g:M\rightarrow H$ is a morphism of monoids, then there exists a unique morphism of groups $h:G\rightarrow H$ such that $g=h\circ f$. In fact, $h([a,b])=g(a)-g(b)$. 

The canonical map $f:M\rightarrow G$ is injective if and only if $M$ has the cancellation property. 
Note that the Grothendieck group of a commutative monoid $M$ is trivial if and only if $M=M^{0}$ where we call $M^{0}=\{x\in M:\exists y\in M, x+y=y\}$ the quasi-zero submonoid of $M$. For example, the Grothendieck group of the multiplicative monoid of natural numbers $\mathbb{N}=\{0,1,2,\ldots\}$ is trivial. 

\section{Grading of homogeneous localization}

Let $R=\bigoplus\limits_{m\in M}R_{m}$ be an $M$-graded ring with $M$ a (commutative) monoid. First recall that every nonzero element $r$ of the additive subgroup $R_{m}$ is called a homogeneous element of $R$ of degree $m\in M$ and written $\dg(r)=m$. If $r,r'\in R$ are homogeneous elements of $R$ with $rr'\neq0$, then $rr'$ is a homogeneous element of $R$ of degree $\dg(rr')=\dg(r)+\dg(r')$.

Next, let $S$ be a multiplicative set of homogeneous elements of $R$. Let $G=\{[a,b]: a,b\in M\}$ be the Grothendieck group of $M$. For each $x\in G$, by $(S^{-1}R)_{x}$ we mean the set of all fractions $f\in S^{-1}R$ such that $f=0$ or it is of the form $f=r/s$ where $r$ is a homogeneous element of $R$ and $x=[\dg(r),\dg(s)]$. 

Then we prove the first main result of this article which asserts that every homogeneous localization of a graded ring is graded by the Grothendieck group with the homogeneous components $(S^{-1}R)_{x}$. To achieve this goal, we first need to treat the cancellative case:

\begin{lemma}\label{Lemma 1-bir} Let $S$ be a multiplicative set of homogeneous elements of an $M$-graded ring $R=\bigoplus\limits_{m\in M}R_{m}$ with $M$ a cancellative monoid. Then $S^{-1}R=\bigoplus\limits_{x\in G}(S^{-1}R)_{x}$ is a $G$-graded ring where $G$ is the Grothendieck group of $M$. 
\end{lemma}

\begin{proof} We have $0\notin S$ and so $st\neq0$ for all $s,t\in S$. We first show that $(S^{-1}R)_{x}$ is an additive subgroup of $S^{-1}R$ for all $x\in G$.
Clearly $(S^{-1}R)_{x}$ is nonempty, because by its definition, $0/1\in(S^{-1}R)_{x}$. Take $f,g\in(S^{-1}R)_{x}$. We will show that $-f$ and $f+g$ are members of $(S^{-1}R)_{x}$. We may assume that $f,g\neq0$. Then $f=r/s$ and $g=r'/s'$ where $r,r'\in R$ are homogeneous elements with 
$[\dg(r),\dg(s)]=x=[\dg(r'),\dg(s')]$. It is clear that $-r$ is homogeneous with $\dg(-r)=\dg(r)$ and so
$[\dg(-r),\dg(s)]=x$. This shows that $-f=(-r)/s\in(S^{-1}R)_{x}$. Both $rs'$ and $r's$ are nonzero and hence are homogeneous. We show that $\dg(rs')=\dg(r's)$. Since $M$ has the cancellation property, the canonical map $M\rightarrow G$ given by $m\mapsto[m,0]$ is injective. The images of both $\dg(rs')=\dg(r)+\dg(s')$ and $\dg(r's)=\dg(r')+\dg(s)$ under this map are the same. Hence, $\dg(rs')=\dg(r's)$.
We may assume $f+g\neq0$. Thus $rs'+r's\neq0$ and so it is homogenous of degree $\dg(rs')$. Then  $[\dg(rs'+r's),\dg(ss')]=[\dg(r)+\dg(s'),\dg(s)+\dg(s')]=
[\dg(r),\dg(s)]+[\dg(s'),\dg(s')]=x$. This shows that $f+g=(rs'+r's)/ss'\in(S^{-1}R)_{x}$. Hence, $(S^{-1}R)_{x}$ is an additive subgroup of $S^{-1}R$. \\
If $f\in(S^{-1}R)_{x}$ and $g\in(S^{-1}R)_{y}$ for some $x,y\in G$, then we show that $fg\in(S^{-1}R)_{x+y}$. We may assume $fg\neq0$. Then $f=r/s$ and $g=r'/s'$ where $r,r'\in R$ are homogeneous elements with 
$[\dg(r),\dg(s)]=x$ and $[\dg(r'),\dg(s')]=y$.
We also have $rr'\neq0$ and so it is a homogeneous element. Thus $[\dg(rr'),\dg(ss')]=
[\dg(r)+\dg(r'),\dg(s)+\dg(s')]=x+y$. \\
It can be easily seen that $S^{-1}R=\sum\limits_{x\in G}(S^{-1}R)_{x}$. 
We need to show that this is a direct sum. Take some $f\in(S^{-1}R)_{x}\cap\sum
\limits_{\substack{y\in G,\\y\neq x}}(S^{-1}R)_{y}$. Suppose $f\neq0$. We may write $f=\sum\limits_{k=1}^{n}f_{k}$ where $0\neq f_{k}\in(S^{-1}R)_{y_{k}}$ for all $k$, and $x\in G\setminus\{y_{1},\ldots,y_{n}\}$. Each $f_{k}=r_{k}/s_{k}$ where $r_{k}$ is a homogeneous element of $R$ with $y_{k}=[\dg(r_{k}),\dg(s_{k})]$. 
We also have $f=r/s$ where $r$ is a homogeneous element of $R$ with $[\dg(r),\dg(s)]=x$.
Then there exists some $t\in S$ such that $rs't=\sum\limits_{k=1}^{n}r_{k}t_{k}st$ where $s':=\prod\limits_{i=1}^{n}s_{i}$ and each $t_{k}:=\prod\limits_{\substack{i=1,\\i\neq k}}^{n}s_{i}$. Note that $rs't$ and each $r_{k}t_{k}st$ are nonzero, and so they are homogeneous elements of $R$.
We claim that if $k\neq d$ for $k,d\in\{1,\ldots,n\}$, then $\dg(r_{k}t_{k}st)\neq\dg(r_{d}t_{d}st)$. Indeed, if $\dg(r_{k}t_{k}st)=\dg(r_{d}t_{d}st)$ then $\dg(r_{k})+\dg(s_{d})+\sum\limits_{\substack{i=1,\\i\neq d,k}}^{n}\dg(s_{i})+\dg(st)=\dg(r_{d})+\dg(s_{k})+
\sum\limits_{\substack{i=1,\\i\neq d,k}}^{n}\dg(s_{i})+\dg(st)$.
Using the cancellation property of $M$, we get that
$\dg(r_{k})+\dg(s_{d})=\dg(r_{d})+\dg(s_{k})$. It follows that $y_{k}=y_{d}$ which is a contradiction (the $y_{k}$ are pairwise distinct). This establishes the claim. Thus $rs't=r_{k}t_{k}st$ for some $k$. It follows that $\dg(r)+\dg(s_{k})+\sum\limits_{\substack{i=1,\\i\neq k}}^{n}\dg(s_{i})+\dg(t)=\dg(s)+\dg(r_{k})+
\sum\limits_{\substack{i=1,\\i\neq k}}^{n}\dg(s_{i})+\dg(t)$. Then using the cancellation property of $M$ once more, we obtain that $\dg(r)+\dg(s_{k})=\dg(s)+\dg(r_{k})$. This yields that $x=y_{k}$ which is a contradiction. Hence, $f=0$. This completes the proof.
\end{proof} 

The Grothendieck group of the additive monoid of natural numbers $\mathbb{N}=\{0,1,2,\ldots\}$ is called the \emph{additive group of integers} and is denoted by $\mathbb{Z}$. This allows us to define the integers (especially the negative integers) in a quite formal way. In fact, $n:=[n,0]$ and so $-n=[0,n]$ and $[m,n]=m-n$ for all $m,n\in\mathbb{N}$. \\

As a first application, the following well known result is recovered:

\begin{corollary}\label{Coro 11 onbir} If $S$ is a multiplicative set of homogeneous elements of an $\mathbb{N}$-graded ring $R=\bigoplus\limits_{n\geqslant0}R_{n}$ then $S^{-1}R=\bigoplus\limits_{n\in\mathbb{Z}}(S^{-1}R)_{n}$ is a $\mathbb{Z}$-graded ring where  $(S^{-1}R)_{n}=\{r/s\in S^{-1}R: r\in R_{n+\dg(s)}\}$ and $R_{d}=0$ for all $d<0$. 
\end{corollary}

\begin{proof} The additive monoid of natural numbers $\mathbb{N}=\{0,1,2,\ldots\}$ has the cancellation property, and hence the assertion follows from Lemma \ref{Lemma 1-bir}.
\end{proof}

To prove the next main result we also need to recall the changing of grading. Let $R=\bigoplus\limits_{m\in M}R_{m}$ be an $M$-graded ring and let $\phi:M\rightarrow M'$ be a morphism of monoids. Then $R=\bigoplus\limits_{d\in M'}S_{d}$ is also an $M'$-graded ring with the homogeneous components $S_{d}=\sum\limits_{\substack{m\in M,\\
\phi(m)=d}}R_{m}$ (it is clear that if $d\in M'$ is not in the image of $\phi$, then $S_{d}=0$).

\begin{remark}\label{Remark iv dort} Let $\{H_{x}\}$ be a family of subgroups of an Abelian group $G$ with the direct sum decomposition $G=\bigoplus\limits_{x}H_{x}$. Assume for each $x$, $S_{x}$ is a subset of $H_{x}$, and every element $f\in G$ can be written as a (finite) sum $f=\sum\limits_{x}f_{x}$ with $f_{x}\in S_{x}$ for all $x$. Then $S_{x}=H_{x}$ for all $x$. Indeed, if $f\in H_{y}$ for some $y$, then $f-f_{y}=\sum\limits_{x\neq y}f_{x}\in H_{y}\cap\sum\limits_{x\neq y}H_{x}=0$. Hence, $f=f_{y}\in S_{y}$.
\end{remark}
 
We are now ready to prove Lemma \ref{Lemma 1-bir} for arbitrary monoids (not necessarily cancellative): 

\begin{theorem}\label{Theorem 20 altin} Let $S$ be a multiplicative set of homogeneous elements of an $M$-graded ring $R=\bigoplus\limits_{m\in M}R_{m}$ with $M$ a monoid. Then $S^{-1}R=\bigoplus\limits_{x\in G}(S^{-1}R)_{x}$ is a $G$-graded ring where $G$ is the Grothendieck group of $M$. 
\end{theorem}

\begin{proof} To prove the assertion, we reduce the problem to the cancellative case and then apply Lemma \ref{Lemma 1-bir}. Let $M'=\{[m,0]: m \in M\}$ be the image of the canonical map $M\rightarrow G$. Then $M'$ is a submonoid of the group $G$ and hence $M'$ is cancellative. It can be easily seen that the Grothendieck group of $M'$ is canonically isomorphic to $G$ (in fact, the Grothendieck group of every submonoid $N$ of $G$ with $M'\subseteq N\subseteq G$ is isomorphic to $G$). By changing the grading, we can view $R=\bigoplus\limits_{x\in M'}T_{x}$ as an $M'$-graded ring with the homogeneous components $T_{x}=\sum\limits_{[m,0]=x}R_{m}$. Every element $s\in S$ is also homogeneous in this new grading of $R$, because $s\in R_{\dg(s)}\subseteq T_{[\dg(s),0]}$. Then by Lemma \ref{Lemma 1-bir}, $S^{-1}R=\bigoplus\limits_{x\in G}B_{x}$ is a $G$-graded ring where each $B_{x}$ is an additive subgroup of $S^{-1}R$ consisting of all fractions $f\in S^{-1}R$ such that $f=0$ or it is of the form $f=r/s$ where $r$ is a homogeneous element of $R$ in this new grading (i.e., $r\in T_{y}$ for some $y=[m,0]\in M'$) and $x=y-[\dg(s),0]=[m,\dg(s)]$. 
For each $x\in G$ the set $(S^{-1}R)_{x}$ is contained in $B_{x}$. It is also clear that each element $f\in S^{-1}R$ can be written as a (finite) sum $f=\sum\limits_{x\in G}f_{x}$ with $f_{x}\in(S^{-1}R)_{x}$ for all $x$. Then by Remark \ref{Remark iv dort}, we have $(S^{-1}R)_{x}=B_{x}$ for all $x\in G$. 
\end{proof}

Note that, under the assumptions of the above result, each homogeneous component $(S^{-1}R)_{x}$ is also the set of all fractions $f=r/s\in S^{-1}R$ such that $f=0$ or $r \in\bigcup\limits_{\substack{m\in M,
\\ [m,\dg(s)]=x}}R_{m}$.

\begin{remark}\label{Remark iii uch} Let $R=\bigoplus\limits_{m\in M}R_{m}$ be an $M$-graded ring with $M$ a  monoid. If $M$ is cancellative, then it can be seen that $1\in R_{0}$
where 0 is the identity element of $M$ (in other words, $R_{0}$ is a subring of $R$). But in general (when $M$ is not necessarily cancellative), this may not happen. In fact in general, the unit of $R$ is an element of the subring $\bigoplus\limits_{m\in M^{0}}R_{m}$ where $M^{0}=\{x\in M:\exists y\in M, x+y=y\}$ is the quasi-zero submonoid of $M$ (even if 1 is homogeneous, it may happen that $\dg(1)\neq0$).
\end{remark}

\begin{corollary}\label{Coro 10 on} Let $S$ be a multiplicative set of homogeneous elements of an $M$-graded ring $R=\bigoplus\limits_{m\in M}R_{m}$ with $M$ a monoid. If $G$ is the Grothendieck group of $M$, then $L=\{[m,\dg(s)]\in G: m\in M, s\in S\}$ is a submonoid of $G$ containing the image of $M$, and $S^{-1}R=\bigoplus\limits_{x\in L}(S^{-1}R)_{x}$ is an $L$-graded ring. 
\end{corollary}

\begin{proof} In fact, $L$ is the set of all elements $x\in G$ which has a presentation of the form $x=[m,\dg(s)]$ with
$m\in M$ and $s \in S$. If $y=[m',\dg(s')]$ is a second element of $L$ then $x+y=[m+m',\dg(s)+\dg(s')]=[m+m',\dg(ss')]\in L$. If 0 denotes the identity element of $M$ then $[0,0]=[\dg(1),\dg(1)]\in L$. Hence $L$ is a submonoid of the group $G$. If $m\in M$ then $[m,0]=[m+\dg(1),\dg(1)]\in L$.
By Theorem \ref{Theorem 20 altin}, $S^{-1}R=\bigoplus\limits_{x\in G}(S^{-1}R)_{x}$ is a $G$-graded ring. But it can be easily seen that $(S^{-1}R)_{x}=0$ for all $x\in G\setminus L$. Hence, $S^{-1}R=\bigoplus\limits_{x\in L}(S^{-1}R)_{x}$ is an $L$-graded ring.  
\end{proof}

\begin{remark} In relation with Corollary \ref{Coro 10 on}, note that the identity element 0 of $M$ is not necessarily a member of $\{\dg(s): s \in S\}$. Although $1\in S$ and so 1 is homogeneous, it may happen that $\dg(1)\neq0$ (see Remark \ref{Remark iii uch}). 
\end{remark}

\begin{proposition} Let $S$ be a multiplicative set of homogeneous elements of an $M$-graded ring $R=\bigoplus\limits_{m\in M}R_{m}$ with $M$ a cancellative monoid. Then  $M'=\{\dg(s): s\in S\}$ is a submonoid of $M$.
\end{proposition}

\begin{proof} Since $M$ is cancellative, then by \cite[Lemma 3.8]{Abolfazl Tarizadeh}, $1\in R_{0}$  where $0$ is the identity element of $M$. Thus $\dg(1)=0\in M'$. If $s,t\in S$ then $st \in S$ and we have $\dg(s)+\dg(t)=\dg(st)\in M'$. Hence, $M'$ is a submonoid of $M$. 
\end{proof}

For any ring $R$ and a monoid $M$, the monoid-ring $R[M]=\bigoplus\limits_{m\in M}S_{m}$ is an $M$-graded ring with the homogeneous components $S_{m}=R\epsilon_{m}=\{r\epsilon_{m}: r \in R\}$ where we call $\epsilon_{m}:=(\delta_{a,m})_{a\in M}$ the monomial of $R[M]$ of degree $m$ and $\delta_{a,m}$ is the Kronecker delta. 

\begin{theorem}\label{Theorem 6 alti} If $S$ is a multiplicative subset of a ring $R$ and $M$ a monoid, then we have the canonical isomorphism of $G$-graded rings $T^{-1}(R[M])\simeq (S^{-1}R)[G]$ with $T=\{s\epsilon_{m}: s\in S, m\in M\}$ and $G$ is the Grothendieck group of $M$.  
\end{theorem}

\begin{proof} By the universal property of monoid-rings, there exists a (unique) morphism of rings $f:R[M]\rightarrow(S^{-1}R)[G]$ which sends each $r\epsilon_{m}$ into $(r/1)\epsilon_{[m,0]}$. But $T$ is a multiplicative set of homogeneous elements of an $M$-graded ring $R[M]$. Thus by Theorem \ref{Theorem 20 altin}, $T^{-1}(R[M])$ is a $G$-graded ring.
The image of each $s\epsilon_{m}\in T$ under $f$ is invertible, because $\big((s/1)\epsilon_{[m,0]}\big)
\big((1/s)\epsilon_{[0,m]}\big)=\epsilon_{[m,m]}=1$. Then by the universal property of localization rings, there exists a (unique) morphism of rings $h:T^{-1}(R[M])\rightarrow(S^{-1}R)[G]$ which sends each $r\epsilon_{m}/s\epsilon_{n}$ into $(r/s)\epsilon_{[m,n]}$. The map $h$ is surjective. It is also a morphism of $G$-graded rings and so its kernel is a graded ideal. Hence, to show that $\Ker h=0$, it suffices to check it for homogeneous elements. If $r\epsilon_{m}/s\epsilon_{n}\in\Ker h$ then $r/s=0$ and so $rs'=0$ for some $s'\in S$. This yields that $r\epsilon_{m}/s\epsilon_{n}=
rs'\epsilon_{m}/ss'\epsilon_{n}=0$. Hence, $h$ is injective.  
\end{proof}

\begin{corollary}\label{Theorem I} For any ring $R$ and a monoid $M$, we have the canonical isomorphism of  
$G$-graded rings $T^{-1}(R[M])\simeq R[G]$ with $T=\{\epsilon_{m}: m\in M\}$ and $G$ is the Grothendieck group of $M$.  
\end{corollary}

\begin{proof} It follows from Theorem \ref{Theorem 6 alti} by taking $S=\{1\}$.
\end{proof}

\begin{remark} If $\{M_{k}: k\in I\}$ is a family of monoids, then the Grothendieck group of the direct sum monoid $\bigoplus\limits_{k\in I}M_{k}$ is canonically isomorphic to the direct sum group $\bigoplus\limits_{k\in I}G_{k}$ where each $G_{k}$ is the Grothendieck group of $M_{k}$.
\end{remark}

\begin{corollary}\label{Corollary I} For any ring $R$ and an index set $I$, the ring of Laurent polynomials $R[x^{\pm1}_{k}: k\in I]$ is canonically isomorphic to the group-ring $R[\bigoplus\limits_{k\in I}\mathbb{Z}]$. 
\end{corollary}

\begin{proof} We know that $R[x_{k}: k\in I]=R[M]$ and by the above remark the Grothendieck group of the additive monoid $M=\bigoplus\limits_{k\in I}\mathbb{N}$ is canonically isomorphic to the additive group $\bigoplus\limits_{k\in I}\mathbb{Z}$. We also have $R[x^{\pm1}_{k}: k\in I]=T^{-1}(R[M])$ where $T=\{\epsilon_{m}: m\in M\}$. Hence, the assertion follows from Corollary \ref{Theorem I}. 
\end{proof}

\section{Cancellative monoids and torsion-free groups}

In the following result, we characterize cancellative monoids by adding two new equivalences (iii) and (iv).  

\begin{proposition}\label{Corollary II} For a  monoid $M$ with the Grothendieck group $G$ the following assertions are equivalent. \\
$\mathbf{(i)}$ $M$ has the cancellation property. \\
$\mathbf{(ii)}$ The canonical map $M\rightarrow G$ is injective. \\
$\mathbf{(iii)}$ For any nonzero ring $R$, the monomial $\epsilon_{m}$ is a non-zero-divisor of $R[M]$ for all $m\in M$. \\
$\mathbf{(iv)}$ For any ring $R$, the canonical ring map $R[M]\rightarrow R[G]$ given by $\sum\limits_{m\in M}r_{m}\epsilon_{m}\mapsto\sum\limits_{m\in M}r_{m}\epsilon_{[m,0]}$ is injective.
\end{proposition}

\begin{proof} (i)$\Leftrightarrow$(ii): It is well known and easy. \\
(i)$\Rightarrow$(iii): If $(\epsilon_{m})f=0$
for some $f=\sum\limits_{x\in M}r_{x}\epsilon_{x}\in R[M]$. Then $\sum\limits_{x\in M}r_{x}\epsilon_{x+m}=0$. Since $M$ is cancellative, $r_{x}\epsilon_{x+m}$ is the ONLY homogeneous element of degree $x+m$ on the left hand side of the above equality. Hence, $r_{x}\epsilon_{x+m}=0$ and so $r_{x}=0$ for all $x\in M$. This shows that $f=0$. \\ 
(iii)$\Rightarrow$(i): Suppose $m+a=m+b$ for some $a,b,m\in M$. Then in $\mathbb{Z}[M]$ we have $\epsilon_{m}(\epsilon_{a}-\epsilon_{b})=0$. By hypothesis, $\epsilon_{m}$ is a non-zero-divisor, thus $\epsilon_{a}=\epsilon_{b}$ and so $a=b$. \\
(i)$\Rightarrow$(iv): Suppose $\sum\limits_{m\in M}r_{m}\epsilon_{[m,0]}=0$. Since $M$ is cancellative, $r_{m}\epsilon_{[m,0]}$ is the ONLY homogeneous element of degree $[m,0]$ on the left hand side of the above equality. 
This yields that $r_{m}\epsilon_{[m,0]}=0$ and so $r_{m}=0$ for all $m\in M$. This shows that $\sum\limits_{m\in M}r_{m}\epsilon_{m}=0$. \\ 
(iv)$\Rightarrow$(i): Suppose $m+a=m+b$ for some $a,b,m\in M$. This shows that $[a,0]=[b,0]$ and so $\epsilon_{a}-\epsilon_{b}$ is in the kernel of the canonical ring map $\mathbb{Z}[M]\rightarrow\mathbb{Z}[G]$. Thus by hypothesis, $\epsilon_{a}=\epsilon_{b}$ and so $a=b$.  
\end{proof}

Recall that by a \emph{totally} (\emph{linearly}) \emph{ordered monoid} we mean a monoid $M$ equipped with a total ordering $<$ such that its operation is compatible with its ordering, i.e. if $a<b$ for some $a,b\in M$, then $a+c\leqslant b+c$ for all $c\in M$. If moreover $M$ is cancellative, then $a<b$ yields that $a+c<b+c$.  

\begin{remark}\label{Remark I} Let $\{M_{i}: i\in I\}$ be a family of totally ordered monoids and let $M=\prod\limits_{i\in I}M_{i}$ be their direct product monoid. Then $M$ can be made into a totally ordered monoid via the lexicographical ordering induced by the orderings on the $M_{i}$. In fact, using the well-ordering theorem (which asserts that every set can be well-ordered), the index set $I$ can be well-ordered. Take $a=(a_{i}), b=(b_{i})\in M$. If $a\neq b$, then the set $\{i\in I: a_{i}\neq b_{i}\}$ is nonempty. Let
$k$ be the least element of this set. Then the lexicographical ordering $<_{\mathrm{lex}}$ is defined on $M$ as $a<_{\mathrm{lex}}b$ or $b<_{\mathrm{lex}}a$, depending on whether $a_{k}< b_{k}$ or $b_{k}< a_{k}$, where $<$ is the ordering on $M_{k}$. Hence, $(M,<_{\mathrm{lex}})$ is a totally ordered monoid. In particular, the direct sum monoid $\bigoplus\limits_{i\in I}M_{i}$ is also a totally ordered monoid, because every submonoid of a totally ordered monoid is itself a totally ordered monoid.
\end{remark}

In the following theorem, we first reformulate the classical result \cite[\S 2.12, Theorem 22]{Northcott} in modern form and then give a new and simple proof for it.   

\begin{theorem}\label{Theorem III} A cancellative monoid is a totally ordered monoid if and only if its Grothendieck group is a torsion-free group. 
\end{theorem}

\begin{proof} Let $M$ be a cancellative monoid with the Grothendieck group $G$. If $M$ is a totally ordered monoid (by an ordering $<$), then it can be easily seen that $G$ is a totally ordered group whose order is defined by $[a,b]<[c,d]$ if $a+d<b+c$ in $M$. It is also easy to check that every totally ordered group is a torsion-free group. As a second proof for the implication ``$\Rightarrow$", suppose some non-identity element $x$ of $G$ is of finite order. So there exists a natural number $n\geqslant2$ such that $nx=0$. We may write $x=[a,b]$ where $a,b\in M$ with $a\neq b$. Then we may assume, say $a<b$. It follows that $na<nb$. But $nx=0$ yields that $[na,nb]=[0,0]$ and so $na=nb$ which is a contradiction. Hence, $G$ is a torsion-free group (i.e. every non-identity element is of infinite order). We prove the reverse implication as follows. Consider $G$ as a $\mathbb{Z}$-module and put $S:=\mathbb{Z}\setminus\{0\}$. Since $G$ is a torsion-free group, the canonical map $G\rightarrow S^{-1}G$ is injective. Note that $S^{-1}G$ is an $S^{-1}\mathbb{Z}$-module. But $\mathbb{Q}:=S^{-1}\mathbb{Z}$ is a field (the field of rational numbers). Hence, the $\mathbb{Q}$-vector space $S^{-1}G$ is isomorphic to a direct sum of copies of $\mathbb{Q}$. The additive group of rational numbers $\mathbb{Q}$ is a totally ordered group whose order is defined by $r/s <r'/s'$ if $rs'<r's$ in $\mathbb{Z}$.
Then by Remark \ref{Remark I}, any direct sum of copies of $\mathbb{Q}$ is a totally ordered group with the lexicographical ordering. Therefore $S^{-1}G$ and hence also $G$ are totally ordered groups. But the canonical map $M\rightarrow G$ is injective. Therefore $M$ is also a totally ordered monoid.
\end{proof}

As an application, Levi's theorem (see \cite[\S3]{Levi} or \cite[Theorem 6.31]{Lam}) is easily deduced:

\begin{corollary}\label{Corollary  IV} Every torsion-free abelian group is a totally ordered group.
\end{corollary}

\begin{proof} The Grothendieck group of every Abelian group is canonically isomorphic to itself. Hence, the assertion immediately follows from Theorem \ref{Theorem III}. 
\end{proof}

\section{Group of units as the Grothendieck group}

The group of units (invertible elements) of a ring $R$ is denoted by $R^{\ast}$. If $S$ is a multiplicative set of a ring $R$, then $S$ is a submonoid of the multiplicative monoid of $R$. We have then the following result:

\begin{lemma}\label{Lemma 2 two-iki} The Grothendieck group of a multiplicative set $S$ of a ring $R$ can be canonically embedded in $(S^{-1}R)^{\ast}$. 
\end{lemma}

\begin{proof} Let $G=\{[s,t]: s,t\in S\}$ be the Grothendieck group of the multiplicative monoid $S$. The map $S\rightarrow(S^{-1}R)^{\ast}$ given by $s\mapsto s/1$ is a morphism of monoids. Then by the universal property of Grothendieck groups, there exists a morphism of groups $f:G\rightarrow(S^{-1}R)^{\ast}$ which sends each $[s,t]$ into $s/t$. If $[s,t]\in\Ker f$, then $ss'=ts'$ for some $s'\in S$. Thus $[s,t]=[1,1]$ is the identity element of $G$. Hence, $f$ is injective. 
\end{proof}

Let $S$ be a multiplicative set of a ring $R$. Then $\overline{S}=\{a\in R:\exists b\in R, ab\in S\}$ is a multiplicative set of $R$ and $S\subseteq\overline{S}$. It 
can be also seen that $\overline{S}=\{a\in R: a/1\in(S^{-1}R)^{\ast}\}=R\setminus\bigcup\limits_
{\substack{\mathfrak{p}\in\Spec(R),\\\mathfrak{p}\cap S=
\emptyset}}\mathfrak{p}$.
The set $\overline{S}$ is called the saturation of $S$, and $S$ is called saturated if $S=\overline{S}$. 
For example, if $\mathfrak{p}$ is a prime ideal of a ring $R$, then $R\setminus\mathfrak{p}$ is saturated. The set of non-zero-divisors of every ring is saturated. The multiplicative set $S=\{4^{n}:n\geqslant0\}=\{1,4,16,\ldots\}$ of $\mathbb{Z}$ is not saturated, because $2\in\overline{S}$ but not in $S$.

\begin{lemma} If $S$ is a saturated multiplicative set of a ring $R$, then the invertible elements of $S^{-1}R$ are precisely of the form $s/t$ with $s,t\in S$.
\end{lemma}

\begin{proof} It is straightforward. 
\end{proof}

\begin{example} The converse of the above lemma does not hold. In fact, we find a non-saturated multiplicative set $S$ in a ring $R$ such that the invertible elements of the nonzero ring $S^{-1}R$ are precisely of the form $s/t$ with $s,t\in S$. The multiplicative set $S=\{ax^{n}: a\in K^{\ast}, n=0 \: or \: n\geqslant2\}$ of the polynomial ring $R=K[x]$ (with $K$ a field) is non-saturated, because $x\in\overline{S}\setminus S$. We know that $K[x]=\bigoplus\limits_{n\geqslant0}Kx^{n}$ is an $\mathbb{N}$-graded ring and so by Corollary \ref{Coro 11 onbir}, $S^{-1}R$ is a $\mathbb{Z}$-graded ring. We know that in a $\mathbb{Z}$-graded integral domain, every invertible element is homogeneous. Thus if $f$ is an invertible element in $S^{-1}R$ then $f=ax^{m}/bx^{n}$ where $bx^{n}\in S$, $a\neq0$ and $m\geqslant0$. We can take $ax^{m}\in S$, because $K$ is a field and so $a\in K^{\ast}$, and if $m=1$ then we may write $f=ax^{m+2}/bx^{n+2}$.
\end{example}

\begin{corollary}\label{Coro 7 yedi} The Grothendieck group of a saturated multiplicative set $S$ of a ring $R$ is canonically isomorphic to $(S^{-1}R)^{\ast}$.
\end{corollary}

\begin{proof} Let $G=\{[s,t]: s,t\in S\}$ be the Grothendieck group of the multiplicative monoid $S$. Then by Lemma \ref{Lemma 2 two-iki}, the canonical morphism of groups $f:G\rightarrow(S^{-1}R)^{\ast}$ given by $[s,t]\mapsto s/t$ is injective. If $r/s\in(S^{-1}R)^{\ast}$ then $rr't=ss't\in S$ where $r'\in R$ and $s',t\in S$. It follows that $r\in\overline{S}=S$ and so $f$ is surjective. 
\end{proof}

\begin{corollary} If $\mathfrak{p}$ is a prime ideal of a ring $R$, then the Grothendieck group of the multiplicative monoid $R\setminus\mathfrak{p}$ is canonically isomorphic to $(R_{\mathfrak{p}})^{\ast}=
R_{\mathfrak{p}}\setminus\mathfrak{p}R_{\mathfrak{p}}$.
\end{corollary}

\begin{proof} It follows from Corollary \ref{Coro 7 yedi}.
\end{proof}

For any ring $R$ by $T(R)$ we mean the total ring of fractions of $R$. In fact, $T(R)=S^{-1}R$ where $S$ is the set of non-zero-divisors of $R$. 

\begin{corollary}\label{coro nice gozel} The Grothendieck group of the set of non-zero-divisors of a ring $R$ is canonically isomorphic to $T(R)^{\ast}$.
\end{corollary}

\begin{proof} It follows from Corollary \ref{Coro 7 yedi}.
\end{proof}

In particular, if $R$ is an integral domain with the field of fractions $F$. Then the Grothendieck group of the multiplicative monoid $R\setminus\{0\}$ is canonically isomorphic to $F^{\ast}=F\setminus\{0\}$. 

\begin{corollary}\label{Coro 8 sekiz} If $S$ is a multiplicative set of a ring $R$, then the Grothendieck group of $\overline{S}$ is canonically isomorphic to $(S^{-1}R)^{\ast}$.
\end{corollary}

\begin{proof} Let $G$ be the Grothendieck group of the multiplicative monoid $T:=\overline{S}$. Since $T$ is saturated, then by Corollary \ref{Coro 7 yedi}, $G$ is canonically isomorphic to $(T^{-1}R)^{\ast}$. But it can be seen that the map $S^{-1}R\rightarrow T^{-1}R$ given by $r/s\mapsto r/s$ is an isomorphism of rings and so it induces an isomorphism between the corresponding groups of units. Hence, $G$ is isomorphic to $(S^{-1}R)^{\ast}$. 
\end{proof}

\begin{corollary} If $I$ is an ideal of a ring $R$, then the Grothendieck group of the multiplicative monoid $T=R\setminus\bigcup\limits_
{\substack{\mathfrak{m}\in\Max(R),\\
I\subseteq\mathfrak{m}}}\mathfrak{m}$ is canonically isomorphic to $(S^{-1}R)^{\ast}$ with $S=1+I$.
\end{corollary}

\begin{proof} We first show that $\overline{S}=T$. It is clear that $\overline{S}\subseteq T$. If $x\in T$ then $x/1$ is invertible in $S^{-1}R$. If not, then $x\in\mathfrak{m}$ for some maximal ideal $\mathfrak{m}$ of $R$ with $\mathfrak{m}\cap S=\emptyset$. If $I$ is not contained in $\mathfrak{m}$ then $\mathfrak{m}+I=R$ and so $\mathfrak{m}\cap S\neq\emptyset$ which is a contradiction. Hence, $I\subseteq\mathfrak{m}$. But this is also contradiction with the choice of $x\in T$. Thus $x/1\in(S^{-1}R)^{\ast}$ and so $x\in\overline{S}$. Now the assertion follows from Corollary \ref{Coro 8 sekiz}.
\end{proof}




\begin{thebibliography}{10}
\bibitem{Lam}
T.Y. Lam, A First Course in Noncommutative Rings, Springer-Verlag, (2001).
\bibitem{Levi}
F.W. Levi, Ordered groups, Proc. Indian Acad. Sci. A \textbf{16}(4) (1942) 256-263.
\bibitem{Northcott}
D.G. Northcott, Lessons on Rings, Modules and Multiplicities, Cambridge University Press (1968).
\bibitem{Abolfazl Tarizadeh}
A. Tarizadeh, Kaplansky's conjectures and homogeneity of zero-divisors, units and idempotents in a graded ring, arXiv.2309.02880 (2024).
\end{thebibliography}
\end{document}